\documentclass[preprint,12pt]{elsarticle}

\usepackage{amsthm,amsmath,amssymb}

\newtheorem{theorem}{Theorem}[section]
\newtheorem{lemma}{Lemma}
\theoremstyle{definition}
\newtheorem{remark}{Remark}
\newtheorem*{note}{Note}

\newcommand{\R}{\mathbb{R}}

\begin{document}
\begin{frontmatter}

\title{Biharmonic nonlinear vector field equations in dimension $4$}

\author[1]{Ioannis Arkoudis\corref{cor1}}
\ead{jarkoudis@math.uoa.gr}
\cortext[cor1]{Corresponding author}

\author[1]{Panayotis Smyrnelis}
\ead{smpanos@math.uoa.gr}

\address[1]{National and Kapodistrian University of Athens, Department of Mathematics, Greece}

\begin{abstract}
Following the approach of Brezis and Lieb \cite{brezis}, we prove the existence of a ground state solution for the biharmonic nonlinear vector field equations in the limiting case of space dimension $4$. Our results complete those obtained by Mederski and Siemianowski \cite{med3} for dimensions $d\ge 5$. We also extend the biharmonic logarithmic Sobolev inequality in \cite{med3} to dimension $4$.
\end{abstract}


\end{frontmatter}

\section{Introduction and main results}
The study of \emph{ground state} problems for the Laplacian operator was mainly motivated by the search of stationary states for some nonlinear evolution equations (e.g. of the Klein-Gordon or Schr\"odinger type), and some other applications in Physics and Biology (e.g. in statistical mechanics \cite{anderson}, cosmology \cite{fra}, or population dynamics \cite{aronson,fife}). 

Given a potential $G:\R^m\to\R$ having a critical point at $0$, and satisfying some growth conditions near $0$ as well as asymptotically, we consider the semilinear system
\begin{align} \label{systemp}
-\Delta u(x)=\nabla G(u(x)), \  u:\R^d\to \R^m, \, d \geq 2, \, m\geq 1,
\end{align}
whose associate energy functional is $S(u)=\int_{\R^d}\big[\frac{1}{2}|\nabla u|^2-G(u)\big]$. The solutions of interest are those converging asymptotically to $0$, and minimizing $S$ among all nonzero solutions in an appropriate class. Such a solution is called a \emph{ground state}. Among the main contributions to solve this problem, we have to mention the seminal paper of Berestycki and Lions \cite{BerLionsI}, which was completed by the work of  Berestycki, Gallou{\"e}t, and Kavian \cite{BGK}, in the two dimensional limiting case, and by the work of Brezis and Lieb \cite{brezis} in the vector case. The method developed in all these papers is based on a constrained minimization. For more recent developments on these results we refer for instance to the papers \cite{alves,clap,med1,med2}. It is also worth noting that in view of the Sobolev embeddings, the case of dimensions $d\geq 3$ differs from the two dimensional limiting case.

Due to the increasing interest for higher order models and their applications (e.g. in nonlinear elasticity \cite{ant} and mechanics \cite{sel}), it is also relevant to investigate ground state problems for the bilaplacian operator $\Delta^2$. This was done by Mederski and Siemienowski \cite{med3} for dimensions $d\geq 5$, and as an application they derived a biharmonic logarithmic Sobolev inequality (cf.  \cite[Theorem 1.4]{med3}).

The scope of the present paper is to study the limiting case of space dimension $4$. Our approach is based on the method introduced in \cite{brezis}, as well as on some tools borrowed from \cite{med3} (e.g. Pohozaev identity in 
\cite[Theorem 1.2]{med3}). Following \cite{med3}, we also extend the biharmonic logarithmic Sobolev inequality in \cite[Theorem 1.4]{med3} to dimension $4$.

In this context, we consider the system
\begin{align} \label{system}
\Delta^2 u(x)=g(u(x)), \  u:\R^4\to \R^m, \, m\geq 1,
\end{align}
where the map $g:\R^m\to\R^m$ is the gradient (except at $u=0$) of a function $G:\R^m\to \R$ satisfying the following assumptions:
\begin{align}
&G(0)=0, \ G:\R^m\to\R \text{ is continuous, and }  G\in C^1(\mathbb{R}^4 \setminus \{0 \} ), \label{g2}\\
&g(0)=0, \text{ and } g(u)=\nabla G(u), \, \forall u\neq 0,\label{g3} \\
&G(u)<0 \ \text{ for } 0< |u| \leq \epsilon , \text{ and for some } \epsilon >0,  \label{g4}\\
&G(u_0)>0 \ \text{for some} \ u_0, \label{g5} \\
&\text{$\forall b >0, \ \forall \delta>0, \ \exists C_{\delta ,b}>0$ such that } |g(u)| \leq C_{\delta,b}+\delta e^{b|u|^{4/3}}, \quad \forall u \in \mathbb{R}^4. \label{g6}
\end{align}
That is, the function $G$ has at $u=0$ a local (but not global) maximum equal to $0$. We also point out that $G$ need not be differentiable at $0$. For instance, we may have $G(u)=-|u|$ near $0$. On the other hand, it follows from (\ref{g6}) that
\begin{align} \label{g7}
&\lim_{|u|\to \infty}\dfrac{|g(u)|}{e^{a|u|^{4/3}}}= 0 \ \text{and} \ \lim_{|u|\to \infty}\dfrac{|G(u)|}{e^{a|u|^{4/3}}}= 0 , \ \forall a>0.
\end{align}  
We shall see later that this growth condition is required to obtain compactness properties from the embedding of $W^{1,4}(\R^4)$ in the Orlicz space $L_A(\R^4)$, where the function $A$ is defined in \eqref{N function}.
 
The energy functional associated with system \eqref{system} is 
\begin{align}
&S(u)=K(u)-V(u) , \label{action} \\
&\text{where } K(u)=\dfrac{1}{2}\int \left\vert D^2u(x) \right\vert^2 dx=\dfrac{1}{2}\sum_{i,j=1}^4\int \Big|\frac{\partial^2 u}{\partial x_i\partial x_j}\Big|^2 dx   , \\
&\text{ and }V(u)=\int G(u(x))dx. \label{potential}
\end{align}
Our goal is to show that $S(u)$ admits a minimum in the set of nontrivial solutions $u\neq 0$ to system \eqref{system}, converging to $0$ as $|x|\to\infty$, in some weak sense. Such a least energy solution is called a ground state. Following \cite{brezis}, we shall look for this solution in the class:
\begin{align*}
\mathcal{C}=\left\lbrace u\in L^1_{loc}(\R^4) \ \middle| \begin{aligned}& \ \nabla u \in D^{1,2}(\mathbb{R}^4), \  G(u)\in L^1(\mathbb{R}^4)   \\  & \ \mu(\{|u|>\eta\})<\infty, \   \forall \eta>0 \end{aligned}\right\rbrace,
\end{align*}
where $\mu$ denotes the Lebesgue measure, and by definition of the space $D^{1,2}$ (cf. \cite{willem}), we have  $$\nabla u \in D^{1,2}(\mathbb{R}^4) \Leftrightarrow \nabla u \in L^4(\mathbb{R}^4), \ D^2u \in L^2(\mathbb{R}^4).$$

It is worth noting that $K(u)=\frac{1}{2}\int_{\R^4}|\Delta u|^2$, provided that $\nabla u \in D^{1,2}(\mathbb{R}^4)$ (cf. Lemma \ref{lap}). Thus,  $D^2u$ is interchangeable with $\Delta u$ in the expression of functional $K$.

\begin{theorem}\label{th1}
If assumptions (\ref{g2})-(\ref{g6}) hold, then
\begin{align} \label{minimizing set}
T=\inf \left\lbrace \int_{\mathbb{R}^4} \dfrac{1}{2}\vert D^2 u \vert^2: \  u\in \mathcal{C} \ , \ \int_{\R^4} G(u)\geq 0 \ , \ u \neq 0 \right\rbrace
\end{align}
is achieved by some $\bar u\in \mathcal{C}$, with $\bar u \neq  0$ and $\int_{\R^4} G(\bar u)=0$, solving (\ref{system}) in $\mathcal D'$. Moreover, we have 
\begin{itemize}
\item $\bar u \in C^{3,\alpha}_{loc}\cap W^{4,q}_{loc}$ for any $0<\alpha<1$ and  $1\leq q<\infty$, 
\item $\lim_{|x|\to\infty} \bar u(x)=\lim_{|x|\to\infty}|\nabla  \bar u(x)|=0$,
\item and the solution $\bar u$ is a ground state, since there holds $0<S(\bar u)\leq S(v)$ for all $v\in \mathcal{C}$ with $v \neq 0$ which are solutions of (\ref{system}) in $\mathcal{D}'$. 
\end{itemize}
\end{theorem}

\begin{remark}If $g$ is locally H\"older continuous, then Theorem \ref{th1} provides a classical solution $\bar u\in C^4(\R^4)$, since the function $x\mapsto g(\bar u(x))$ is locally H\"older continuous.  
\end{remark}

As an application of Theorem \ref{th1}, we choose the potential 
\begin{equation*}
G(u)=\begin{cases}|u|^2\ln|u| &\text{ for } u\neq 0\\
0&\text{ for } u=0,
\end{cases}\end{equation*}which clearly satisfies assumptions (\ref{g2})-(\ref{g6}). Moreover, for this choice of $G$, one can see that $\mathcal C=\{u\in H^2(\R^4;\R^m): \, G^-(u)\in L^1(\R^4)\}$ (cf. Remark \ref{remf}). Following \cite{med3}, we derive in \eqref{sobin} below the biharmonic logarithmic Sobolev inequality in dimension $4$ (cf. \cite[Theorem 1.4]{med3} for dimensions $d\geq 5$). It is a generalization for the bilaplacian operator of  the classical logarithmic Sobolev inequality 
\begin{align}\label{sobin2}
\ln\left(\dfrac{1}{2\pi e}\int_{\R^4}|\nabla u|^2\right)\geq \int_{\R^4}|u|^2\ln|u| ,\text{ for } u\in H^1(\R^4), \int_{\R^4}|u|^2=1,
\end{align}
given in \cite{weiss}, which is optimal in the scalar case (cf. \cite{carl}). 
\begin{theorem}\label{th2} 
For any $u\in H^2(\R^4;\R^m)$ such that $\int_{\R^4}|u|^2=1$, there holds
\begin{align}\label{sobin}
\dfrac{1}{2} \ln\left(\dfrac{1}{2T}\int_{\R^4}|\Delta u|^2\right)=\dfrac{1}{2} \ln\left(\dfrac{1}{2T}\int_{\R^4}|D^2u|^2\right)\geq 
\int_{\R^4}|u|^2\ln|u| ,
\end{align}
with $\frac{1}{2T}<\frac{1}{(2\pi e)^2}$, where $T>0$ is the minimum in \eqref{minimizing set}. Moreover, the equality in \eqref{sobin} holds provided that $u = \bar u/\|\bar u\|_{L^2(\R^4)}$ and $ \bar u$ is a ground state solution to \eqref{system}. Conversely, if the equality in \eqref{sobin} holds for $u$, then there are uniquely determined $\mu > 0$ and $r>0$, such that $\bar u(x) = 
\mu u(rx)$ is a ground state solution to \eqref{system}.
\end{theorem}
\begin{remark}
If $u\in H^2(\R^4;\R^m)$, we have $(|u|^2\ln |u|)^+\in L^1(\R^4)$ (cf. Remark \ref{remf}), and hence the right hand side of \eqref{sobin} is well defined, although it could be $-\infty$ if  $\int_{\R^4}(|u|^2\ln |u|)^-=+\infty$. 
\end{remark}

\begin{remark}
Equality case in \eqref{sobin} is achieved for the normalized ground state solutions provided by Theorem \ref{th1}. However, these solutions are not known explicitly, thus we neither know the exact value of the optimal constant in \eqref{sobin}. By comparing it with the optimal constant (cf. \cite{carl}) of the classical logarithmic Sobolev inequality \eqref{sobin2}, we can at least obtain that $\frac{1}{2T}<\frac{1}{(2\pi e)^2}$.
\end{remark}

Logarithmic Sobolev type inequalities appear in several equivalent formulations. The seminal inequality by Gross \cite{Gross} was established in the Gaussian setting. Weissler's approach \cite{weiss} gives the Lebesgue analogue with sharp constant which under the unit normalization
$\|u\|_{L^{2}(\mathbb R^{4})}=1$ leads to the form (\ref{sobin2}) (see also \cite{lieb book}).

In a series of papers \cite{CT02,CT04}, Cotsiolis and Tavoularis proved sharp Sobolev type inequalities with explicit best constants for higher order fractional derivatives  in the range $d>2s$.
Most relevant here is their higher order fractional derivative logarithmic result \cite{CT05}, which for every $s>0$ and every $\alpha>0$ gives a sharp
inequality with explicit optimal constants: for $u\in H^{s}(\mathbb R^{d})$,
\begin{multline}\label{Cotsiolis Tavoularis inequality}
\int_{\mathbb R^{d}} |u|^{2}\ln\!\Big(\frac{|u|^{2}}{\|u\|_{2}^{2}}\Big)\,dx
+\Big(d+\frac{d}{s}\ln\alpha+\ln\!\Big(s\,\frac{\Gamma(d/2)}{\Gamma(d/(2s))}\Big)\Big)\|u\|_{2}^{2}\\
\;\le\;\frac{\alpha^{2}}{\pi^{s}}\|(-\Delta)^{s/2}u\|_{2}^{2}.
\end{multline}
In particular, for $s=1$ one recovers
the Gross inequality (\ref{sobin2}) by setting $\|u\|_{2}=1$ and minimizing in $\alpha$.
\par 
Finally, Mederski and Siemianowski (for $d\ge 5$) derived a scalar biharmonic logarithmic Sobolev inequality tied to the bilaplacian, i.e. for $u\in D^{2,2}(\R^d)$ with $\|u\|_{2}=1$,
\begin{equation} \label{Mederski inequality}
\int_{\mathbb R^{d}} |u|^{2}\ln|u|\,dx \;\le\; \frac{d}{8}\ln\!\Big(C\int_{\mathbb R^{d}}|\Delta u|^{2}\,dx\Big),
\end{equation}
with some $0<C<\big(\frac{2}{\pi e d}\big)^{2}$. Moreover, they relate the equality to normalized ground state solutions of
the associated biharmonic problem and note that the exact optimal constant is not explicitly computable. Note that, the case $s=2$ in (\ref{Cotsiolis Tavoularis inequality}) corresponds to the logarithmic inequality controlled by $\|\Delta u\|_{2}^{2}$, matching the inequality above.

The proofs of Theorems \ref{th1} and \ref{th2} are given in Sections \ref{sec:sec2} and \ref{sec:sec3} respectively, while in Section \ref{sec:sec4} we gather some technical lemmas.

\section{Proof of Theorem \ref{th1}}\label{sec:sec2}
Let $\{ u_j \}\subset \mathcal C$ be a minimizing sequence for (\ref{minimizing set}), i.e. $$\lim_{j\to\infty}\dfrac{1}{2}\int_{\R^4}|D^2 u_j|^2 =T.$$  Since in the limiting case of dimension $4$, $K(u)$ is invariant under scaling, we may assume that 
\begin{align} \label{measure is 1}
\mu(\{|u_j|>\varepsilon /2\})=1, \text{  where $\epsilon$ is defined in \eqref{g4}}.
\end{align}
It is clear that the norms  $\Vert D^2 u_j \Vert_{L^2(\mathbb{R}^4)}$ are uniformly bounded. Therefore, due to the Sobolev imbedding of $D^{1,2}(\mathbb{R}^4)$ in $L^4(\mathbb{R}^4)$ (cf. \cite[Thm. 1.8]{willem}), we also obtain that the norms $\Vert \nabla u_j \Vert_{L^4(\mathbb{R}^4)}$ are uniformly bounded.
\par Next, we show that  $\{ u_j \}\subset L^q_{loc}(\R^4)$, $\forall  q \in [1,+\infty)$, and estimate the $L^q$ norm of $u_j$.
\begin{lemma}\label{Lemma 3.1}
We have that
\begin{align}\label{inequality of lemma 3.1}
\int_{|u_j|>\varepsilon /2}|u_j|^q \leq C_q , \ \forall q  \in [1,+\infty) , \ \forall j \text{ and for some constant } C_q.
\end{align}
\end{lemma}
\begin{proof}
Our first claim is that given $\phi \in L^1_{loc}(\mathbb{R}^4)$, such that $\nabla \phi \in L^4(\mathbb{R}^4)$ and $\mu (supp\, \phi)<\infty$, we have
\begin{align}\label{blq}
\Vert \phi \Vert_{L^q(\R^4)} &\leq k_q \Vert \nabla \phi \Vert_{L^4(\R^4)} (\mu (supp\, \phi ))^{\frac{1}{q}}, \ \forall q \geq\frac{4}{3},
\end{align}
and for a constant $k_q>0$. To establish \eqref{blq}, we follow the steps of the proof of \cite[inequality (3.10)]{brezis}. Indeed, it holds (cf. \cite[Theorem p. 125]{nir0}):
\begin{align}\label{blqaa}
\Vert \phi \Vert_{L^\frac{4}{3}(\R^4)} &\leq k \Vert \nabla \phi \Vert_{L^1(\R^4)},\ \forall \phi\in C^1_0(\R^4),
\end{align}
for a constant $k>0$. Next, by applying \eqref{blqaa} to the functions $\phi^n$, $n=2,3,\ldots$, we get by induction 
\begin{align}\label{blqa}
\Vert \phi \Vert_{L^\frac{4n}{3}(\R^4)} &\leq k_n \Vert \nabla \phi \Vert_{L^1(\R^4)}^{\frac{1}{n}} \Vert \nabla \phi \Vert_{L^4(\R^4)}^{\frac{n-1}{n}},
\end{align}
for a constant $k_n>0$, and by interpolation
\begin{align}\label{blqaaa}
\Vert \phi \Vert_{L^q(\R^4)} &\leq k_q \Vert \nabla \phi \Vert_{L^1(\R^4)}^{\frac{4}{3q}} \Vert \nabla \phi \Vert_{L^4(\R^4)}^{1-\frac{4}{3q}}, \ \forall q\geq \frac{4}{3},
\end{align}
for a constant $k_q>0$. Moreover, smoothing by convolution, we notice that \eqref{blqaaa} still holds for functions $\phi\in L^\infty(\R^4)$, such that $\nabla \phi \in L^4(\R^4)$, and $\phi$ has compact support. On the other hand, a standard cut-off and truncation argument shows the validity of \eqref{blqaaa} for functions 
$\phi\in L^1_{loc}(\R^4)$, such that $\nabla \phi \in L^4(\R^4)$, and $\mu(supp\, \phi)<\infty$.
Finally, \eqref{blq} follows from \eqref{blqaaa} by H{\"o}lder's inequality.

To conclude the proof of the lemma, we apply \eqref{blq} to $\phi =(|u_j|-\varepsilon /2 )^+$,  and obtain 
\begin{align*}
\Vert(|u_j|-\varepsilon /2 )^+\Vert_{L^q(\R^4)}\leq k_q \Vert \nabla u_j \Vert_{L^4(\R^4)}, 
\end{align*}
which implies \eqref{inequality of lemma 3.1}.
\end{proof} 

At this stage, we recall the following compactness result \cite[Lemma 6]{lieb}.
\begin{lemma}\label{shift lemma}
Let $v \in L^4_{loc}(\mathbb{R}^4)$, $\nabla v \in L^4(\mathbb{R}^4)$, $\Vert \nabla v \Vert_{L^4} \leq C$ and $\mu(\{|v|>\varepsilon / 2\}) \geq \delta >0$. Then, there exists a shift $T_yv(x)=v(x+y)$ such that, for some constant $\beta=\beta(C,\delta, \varepsilon)>0$, $\mu(B\cap \{|T_yv|>\varepsilon /4\})>\beta$, where $B=\{x\in\R^4:|x|\leq 1\}$.
\end{lemma}

In view of Lemma \ref{shift lemma}, we can shift each $u_j$ in such a way that  $\mu(B\cap \{|T_{x_j}u_j|>\epsilon/4\} )\geq \beta$. Thus, we may assume that  $\mu(B\cap \{|u_j|>\epsilon/4\} )\geq \beta$. On the other hand, Lemma \ref{Lemma 3.1} implies that given a bounded set $Q\subset\R^4$, the norms $\Vert  u_j \Vert_{L^q(Q)}$ are uniformly bounded. Consequently, we can extract a subsequence still called $\{u_j\}$ such that
\begin{enumerate}
\item[(i)] $u_j\to u$ in $L^q(Q)$, $\forall q\in [1,+\infty)$, $\forall Q\subset\R^4$ bounded,
\item[(ii)] $u_j\to u$ a.e. in $\mathbb{R}^4$,
\item[(iii)] $\nabla u_j \rightharpoonup \nabla u$ in $L^4(\mathbb{R}^4)$,
\item[(iv)] $D^2 u_j \rightharpoonup D^2 u$ in $L^2(\mathbb{R}^4)$,
\item[(v)] $\mu(B\cap \{|u| \geq \varepsilon /4\})\geq \beta >0 \ \Rightarrow u \not\equiv 0$,
\item[(vi)] $\mu(\{|u| >\varepsilon / 2\})\leq 1$,
\item[(vii)] $u \in \mathcal C$,
\end{enumerate}
Properties (v) and (vi) of the limit $u$ follow from the a.e. convergence in (ii). To check that $G(u)\in L^1$, we use assumption (\ref{g7}) and the embedding
\begin{align} \label{Orlicz embedding}
W^{1,4}(\R^4) \hookrightarrow L_A(\R^4), 
\end{align}
where $L_A$ is the Orlicz space corresponding to the function
\begin{align}\label{N function}
A(t)=\exp(t^{4/3})-1-t^{4/3}-\dfrac{1}{2}t^{8/3}
\end{align}
(cf. Trudinger's theorem in \cite[\S 8.29]{adams}). Since the sequence $\{(|u_j|-\epsilon/2 )^+\}$ is uniformly bounded in $W^{1,4}(\R^4)$, it follows that there exists $k_1>0$, such that the integrals $\int_{\R^4}A(k_1(|u_j|-\epsilon/2 )^+)$ are uniformly bounded. Moreover, it holds $G^+(u_j)\leq k_2 A(k_1(|u_j|-\epsilon/2 )^+)$, for a constant $k_2>0$ (cf. \eqref{g4} and \eqref{g7}). Consequently, the integrals $\int_{\R^4}G^+(u_j)$ are uniformly bounded, and we deduce from Fatou's lemma that $\int_{\R^4}G^+(u)<\infty$.
Similarly, since we have $\int_{\R^4}G^-(u_j)\leq \int_{\R^4}G^+(u_j)$, the integrals $\int_{\R^4}G^-(u_j)$ are uniformly bounded, and we obtain as previously that  $\int_{\R^4}G^-(u)<\infty$. Therefore, $G(u)\in L^1$. To prove that $u\in\mathcal C$, it remains to show that $\mu(\{|u|>\alpha\})<\infty$, $\forall \alpha>0$. Indeed, if $\alpha\geq \epsilon/2$, this is clear in view of (vi) above. Otherwise, if $0<\alpha<\epsilon/2$, we have in view of hypothesis \eqref{g4}:
$$\mu(\{\alpha<|u|\leq \epsilon/2\}) \min\{|G(v)|: \alpha\leq|v|\leq\epsilon/2\}\leq \int_{\alpha \leq|u|\leq\epsilon/2}|G(u)| <\infty,$$
thus $\mu(\{\alpha<|u|\leq \epsilon/2\})<\infty\Rightarrow \mu(\{|u|>\alpha\})<\infty$. This completes the proof of property (vii) above.

Following \cite{brezis}, we also define the class
\begin{align}\label{new class K}
\mathcal{K} :=\left\lbrace \phi :\R^4\to\R^m \ | \  \phi\in L^1_{loc}, \ \nabla \phi \in L^4  ,  \ D^2\phi \in L^2, \ \mu (supp\, \phi )<\infty \right\rbrace.
\end{align}
\begin{remark} Due to \eqref{blq} we have that $\mathcal K\subset L^q(\mathbb{R}^4)$, $\forall q\in[1,\infty)$.
\end{remark} 

In the sequel, we proceed  in several steps as in \cite{brezis}, to construct the solution of Theorem \ref{th1}.

\begin{lemma}\label{lemma corresponding to 3.2}
Let $\phi \in \mathcal{K}$ be such that
\begin{align}\label{gpos}
\int_{\R^4}( G(u+\phi)-G(u))>0.
\end{align}
Then,
\begin{align}\label{limeq2}
\int_{\R^4} D^2 u \cdot D^2 \phi + \dfrac{1}{2}\int_{\R^4} \vert D^2\phi \vert^2 \geq 0, \text{ where } D^2 u \cdot D^2 \phi=\sum_{i,j=1}^4\partial^2_{ij} u\cdot \partial^2_{ij} \phi.
\end{align}
\end{lemma}
\begin{note}
We shall see in the proof of Lemma \ref{lemma corresponding to 3.2} that $G(u+\phi)\in L^1(\R^4)$, $\forall \phi \in \mathcal K$. 
\end{note}

\begin{proof}
Let $\phi\in \mathcal K$. We first notice that the sequence $\{(|u_j+\phi|-\epsilon/2)^+\}$ is uniformly bounded in $L^4(\R^4)$. Indeed, this follows  from the inequality
\begin{align*}
\int_{|u_j+\phi|>\epsilon/2}|u_j+\phi|^4&\leq \int_{|u_j|>\epsilon/2}|u_j+\phi|^4+\int_{supp \, \phi \cap\{|u_j|\leq\epsilon/2\}} |u_j+\phi|^4,
\end{align*}
combined with \eqref{inequality of lemma 3.1}, and the inclusion  $\mathcal K\subset L^4(\mathbb{R}^4)$. Moreover, one can see that the sequence $\{(|u_j+\phi|-\epsilon/2)^+\}$ is also uniformly bounded in $W^{1,4}(\R^4)$. Therefore, there exists $k_1>0$, such that the integrals $\int_{\R^4}A(2 k_1(|u_j+\phi|-\epsilon/2 )^+)$ are uniformly bounded. Next, we check that the sequence $\{G(u_j+\phi)\}$ is uniformly integrable on $\R^4$.
To prove this, we use the bounds $|G(u)|\leq k_2+ A(k_1(|u|-\epsilon/2 )^+)$,  and $|A(|u|)|^2\leq k_3+A(2|u|)$, holding in $\R^m$ for some constants $k_2,k_3>0$. Then, given $\omega\subset \R^4$, with $\mu(\omega)<\infty$, we obtain
\begin{align*}
\int_{\omega}|G(u_j+\phi)|&\leq k_2 \mu(\omega)+\int_{\omega}A( k_1(|u_j+\phi|-\epsilon/2 )^+)\\
&\leq k_2 \mu(\omega)+(\mu(\omega))^\frac{1}{2}\Big( \int_{\omega}|A( k_1(|u_j+\phi|-\epsilon/2 )^+)|^2\Big)^\frac{1}{2}\\
&\leq k_2 \mu(\omega)+(\mu(\omega))^\frac{1}{2}\Big(k_3 \mu(\omega)+\int_{\R^4}A(2 k_1(|u_j+\phi|-\epsilon/2 )^+)\Big)^\frac{1}{2},
\end{align*}
which implies the uniform integrability of $\{G(u_j+\phi)\}$ on $\R^4$. Consequently, since $G(u_j+\phi)\to G(u+\phi)$ a.e. in $\R^4$, we deduce that $G(u+\phi)$ is integrable on $\omega=supp\, \phi$, as well as on $\R^4$. In addition, we have $\lim_{j\to\infty}\int_{supp\, \phi}G(u_j+\phi)=\int_{supp\, \phi}G(u+\phi)$. The same argument also proves that the sequence $\{G(u_j)\}$ is uniformly integrable on $\R^4$, and that $\lim_{j\to\infty}\int_{supp\, \phi}G(u_j)=\int_{supp\, \phi}G(u)$. Therefore, we have $\lim_{j\to\infty}\int_{\R^4}(G(u_j+\phi)-G(u_j))=\int_{\R^4}(G(u+\phi)-G(u))$.

Finally, assuming \eqref{gpos}, $u_j+\phi \in \mathcal C$ satisfies $\int_{\R^4}(G(u_j+\phi)-G(u_j)) > 0$, as well as $\int_{\R^4}G(u_j+\phi) > 0$, for $j$ large enough.
Thus, it follows that
\begin{equation}\label{limeq}
\frac{1}{2}\int_{\R^4}|D^2 (u_j+\phi)|^2=\frac{1}{2}\int_{\R^4}|D^2 (u_j)|^2+\int_{\R^4}D^2 u_j \cdot D^2\phi+\frac{1}{2}\int_{\R^4}|D^2 \phi|^2\geq T.
\end{equation}
Since, $\lim_{j\to\infty}\frac{1}{2}\int_{\R^4}|D^2 (u_j)|^2=T$, by letting $j\to\infty$ in \eqref{limeq}, we get \eqref{limeq2}.
  \end{proof}

\begin{lemma}\label{lem2}
There exists a constant $ C_G $ (depending only on $ G $) such that if $ \phi \in \mathcal{K} $, then
\begin{align}\label{eqlem2}
\int_{\R^4} g(u) \cdot \phi - C_G \int_{u=0} |\phi| > 0 \ \Rightarrow \ \int_{\R^4} D^2 u \cdot D^2 \phi \geq 0  . 
\end{align}
\end{lemma}

\begin{note}
We shall see in the proof of Lemma \ref{lem2} that $g(u)\cdot \phi\in L^1(\R^4)$, $\forall \phi \in \mathcal K$. 
\end{note}
\begin{proof} Let $\phi \in \mathcal K$. By dominated convergence, one can see that $g(u) \cdot \phi \in L^1(\R^4)$, and as \( t \to 0^+ \), we have
\begin{equation}\label{cvdom}
\int_{u\neq 0} \big( G(u+t\phi) - G(u) \big)  = t \int_{u\neq 0} g(u) \cdot \phi  + o(t)=t \int_{\R^4} g(u) \cdot \phi + o(t)  .
\end{equation}
Indeed, it is clear that $\lim_{t\to 0}\frac{ G(u(x)+t\phi(x)) - G(u(x)) }{t}=g(u(x))\cdot \phi (x)$, provided that $u(x)\neq 0$. On the other hand, since $(|u|-\epsilon/2)^++|\phi|\in W^{1,4}(\R^4)$ (cf. \eqref{blq}), there exists $k_1>0$, such that $\int_{\R^4}A(2 k_1[(|u|-\epsilon/2 )^++|\phi|])<\infty$.  Next, in view of the bound $|g(u)|\leq k_2+ A(k_1(|u|-\epsilon/2)^+)$ and the monotonicity of $A$, we deduce that 
\begin{align*}
\forall s,t \in [0,1]: \ |g(u+st\phi)|&\leq  k_2+ A(k_1(|u+st\phi|-\epsilon/2)^+)\\
&\leq k_2+ A(k_1[(|u|-\epsilon/2)^++|\phi|]).
\end{align*}
Consequently, we obtain the inequality
\begin{multline*}
\Big|\frac{ G(u(x)+t\phi(x)) - G(u(x)) }{t}\Big|=\Big|\int_0^1 \big(g(u(x)+st\phi(x))\cdot \phi(x)\big)ds\Big|  \\ 
\leq  k_2|\phi(x)|+  A(k_1[(|u(x)|-\epsilon/2)^++|\phi(x)|])|\phi(x)|,
\end{multline*}
whose right hand side is an integrable function (since $A(k_1[(|u(x)|-\epsilon/2)^++|\phi(x)|])\in L^2(supp\, \phi)$ in view of the bound  $|A(|u|)|^2\leq k_3+A(2|u|)$). This proves that $g(u) \cdot \phi \in L^1(\R^4)$ and \eqref{cvdom} holds.

Our next claim is that
\begin{equation}\label{claim1}
\left| \int_{u=0} G(t\phi)  \right| \leq C_G  t\int_{u=0} |\phi| + o(t), \text{ as } t\to 0^+.
\end{equation}

To justify this estimate, we notice that given $\phi  \in \mathcal{K} \subset W^{1,4}(\R^4)$, there  exists $k_4>0$ such that $A(k_4|\phi|)\in L^1(\R^4)$. In view of the bound $|G(u)|\leq C_G|u|+k_5 A(k_4 |u|)$, holding in $\R^m$ for a constant $C_G>0$ depending only on $G$, and a constant $k_5>0$ depending also on $k_4$ (cf. \eqref{g6}), it follows that
\begin{equation*}
\left| \int_{u=0} G(t\phi)  \right| \leq C_G  t\int_{u=0} |\phi| + k_5 \int_{u=0} A(k_4t|\phi|). 
\end{equation*}
Finally, using the fact that $ A(k_4t|\phi|)\leq t^4 A(k_4|\phi|)$ holds for $0<t<1$, we deduce \eqref{claim1}.

Now, choose $\phi \in \mathcal{K}$ such that
\[
\int_{\R^4} g(u) \cdot \phi - C_G \int_{u=0} |\phi| > 0.
\]
Then, for small enough \( t > 0 \), it follows from \eqref{cvdom} and \eqref{claim1} that
\[
\int_{\R^4} \big( G(u+t\phi) - G(u) \big) > 0.
\]
By Lemma \ref{lemma corresponding to 3.2}, we conclude that
\[
\int_{\R^4} D^2 u \cdot D^2 \phi + \dfrac{t}{2}\int_{\R^4} \vert D^2\phi \vert^2 \geq 0, \  \text{ for  } 0<t\ll 1\Rightarrow \int_{\R^4} D^2 u \cdot D^2 \phi \geq 0.
\]
\end{proof}

\begin{lemma}\label{lem4}
Consider the linear functional $L(\phi ) = \int_{\R^4} D^2 u \cdot D^2 \phi$. There is some $\phi_0 \in \mathcal{K}$ such that $L(\phi_0) \neq 0$ and $\phi_0 =0 $ on the set $\{u=0\}$.
\end{lemma}
\begin{proof}
Let $\phi_n =\rho_n u  \in \mathcal{K}$, where $\rho_n=\rho(x/n)$, and $\rho\in C^\infty_0(\R^4)$ is a function such that $0\leq \rho\leq 1$ and  
\begin{align}
\rho (x) =\begin{cases} 1  &\text{ for } |x|\leq 1 \\
 0 &\text{ for } |x|\geq 2.\end{cases}
\end{align}
Let also $A_n=\{ x\in\R^4: n\leq |x|\leq 2n\}$. We assume by contradiction that $L(\phi)=0$, for all $ \phi\in\mathcal  K$ such that $\phi =0 $ on  $\{u=0\}$. It is clear that $\phi_n =0$ on the set $\{u=0\}$.  Thus, we have $L(\phi_n)=L(\rho_n u)=0$. That is,
\begin{align*}
\int_{\R^4}|D^2u|^2\rho_n+\sum_{i,j=1}^4\big(\partial^2_{ij} u\cdot\big(\partial_{j}\rho_n\,\partial_{i}u+\partial_{i}\rho_n\,\partial_{j}u\big)+\big(u\cdot\partial^2_{ij}u\big)\partial^2_{ij} \rho_n\big)=0.
\end{align*}
After direct calculations we obtain that $$\lim_{n\to\infty}\int_{\R^4}|D^2u|^2\rho_n= \int_{\R^4}|D^2u|^2,$$
$$\lim_{n\to\infty}\int_{\R^4}\big|\big(\partial^2_{ij}u\cdot \partial_i u\big)\partial_j \rho_n\big|\leq k_\rho\|D^2u\|_{L^2(A_n)}\|\nabla u\|_{L^4(A_n)}\to 0 \text{, as } n\to\infty,$$
\begin{align*}
\lim_{n\to\infty}\int_{\R^4}\big|\big(u\cdot\partial^2_{ij} u\big)\partial^2_{ij} \rho_n\big|&\leq \frac{k_\rho}{n}\|D^2u\|_{L^2(A_n)}\| u\|_{L^4(A_n)}\\
&\leq \frac{k_\rho}{n}\|D^2u\|_{L^2(A_n)}\big(\| (|u|-\frac{\epsilon}{2})^+\|_{L^4(A_n)}+\frac{\epsilon}{2}\mu(A_n)^{\frac{1}{4}}\big)\\
&\to 0, \text{  as $n\to\infty$ (cf. \eqref{blq})},
\end{align*}
where $k_\rho>0$ is a constant depending on $\rho$.
This implies that $D^2 u \equiv 0$ i.e. $u$ is affine, and since $\mu(\{|u|>\eta\})<\infty$, $\forall \eta>0$, we get $u\equiv 0$, which is a contradiction.
\end{proof}

Lemmas 3.5 from \cite{brezis} is applicable to our case without modifications. For convenience of the reader we reproduce in Lemmas \ref{lem5} below its proof.
\begin{lemma}\label{lem5}
There is a constant $\lambda\geq 0$ such that
\begin{equation}\label{eqlem5}
\Big|\int_{\R^4}g(u)\cdot \phi-\lambda L(\phi)\Big |\leq C_G\int_{u=0}|\phi|, \ \forall \phi\in \mathcal K.
\end{equation}
\end{lemma}
\begin{proof}
In view of Lemma \ref{lem4}, we can choose $\phi_0\in\mathcal K$ such that $\phi_0=0$ on $\{u=0\}$, and $L(\phi_0)=-1$.
Next, given $\phi\in\mathcal K$, we consider the function $\psi=\phi+L(\phi)\phi_0+\alpha\phi_0$, with $\alpha>0$. It is clear that $L(\psi)=-\alpha<0$, thus by Lemma \ref{lem2}, we have that $\int_{\R^4} g(u) \cdot (\phi+L(\phi)\phi_0+\alpha\phi_0) - C_G \int_{u=0} |\phi| \leq 0.$
Setting $\lambda=-\int_{\R^4} g(u) \cdot \phi_0\geq 0$ (by \eqref{eqlem2}), and letting $\alpha\to 0$, we deduce that $\int_{\R^4}g(u)\cdot \phi-\lambda  L(\phi)- C_G\int_{u=0}|\phi|\leq 0$, $\forall \phi\in \mathcal K$, from which \eqref{eqlem5} follows.
\end{proof}

\begin{lemma}\label{lem6}
We have that 
\begin{equation}\label{eqdi}
\frac{\partial[G(u)]}{\partial x_i}=g(u)\cdot\frac{\partial u}{\partial x_i} \text{  in  }\mathcal D'.
\end{equation}
\end{lemma}

\begin{note}
It is clear that $G(u)\in L^1(\R^4)$. On the other hand, we shall see in the proof of Lemma \ref{lem6} that $g(u)\cdot\frac{\partial u}{\partial x_i}\in L^1_{loc}(\R^4)$, thus \eqref{eqdi} makes sense. 
\end{note}

\begin{proof}

As in \cite[Lemma 3.6]{brezis}, we choose a smoothing sequence $\{G_j\}\subset C^\infty(\R^m)$ for $G$, satisfying the following properties:
\begin{itemize}
\item[(i)] $G_j\to G$, as $j\to\infty$, pointwise on $\R^m$,
\item[(ii)] $g_j=\nabla G_j\to g$, as $j\to\infty$,  pointwise on $\R^m\setminus\{0\}$,
\item[(iii)] $\forall k_2>0$, $\exists k_3>0$ such that the inequalities $|g_j(u)|\leq k_3+e^{k_2|u|^{\frac{4}{3}}}$, and $|G_j(u)|\leq k_3+e^{k_2|u|^{\frac{4}{3}}}$, hold on $\R^m$ for all $j$.
\end{itemize}
Next, given an arbitrary open ball $B\subset\R^4$, we can find a sequence $\{v_n\}\subset C^\infty(\R^4)$, such that $v_n$ converges to $u$ in $W^{1,4}(B)$, while $v_n$ and $\nabla v_n$ converge almost everywhere respectively to $u$ and $\nabla u$. In view of the Trudinger inequality \cite[\S 8.27]{adams}, there exists a constant $k_1>0$ such that the functions $e^{k_1|v_n|^{\frac{4}{3}}}$ are uniformly bounded in $L^1(B)$. Choosing $k_2=\frac{k_1}{2}$, it follows from (iii) above, that the functions $g_j(v_n)$ and $G_j(v_n)$ are uniformly bounded in $L^2(B)$. 
 Consequently, given a test function $\phi\in C^\infty_0(B)$ and $i\in\{1,\ldots,4\}$, one can see that the sequences $\{G_j(v_n)\partial_i\phi\}_n$ and $\{(g_j(v_n)\cdot\partial_i v_n)\phi \}_n$ are uniformly integrable over $B$. Since there holds $$\int_B G_j(v_n)\partial_i\phi=-\int_B(g_j(v_n)\cdot\partial_i v_n)\phi ,$$
  by letting $n\to\infty$, we deduce form the Vitali convergence theorem that
  $$\int_B G_j(u)\partial_i\phi=-\int_B(g_j(u)\cdot\partial_i u)\phi .$$
Finally, by letting $j\to\infty$, and using the uniform integrability over $B$ of the sequences  $\{G_j(u)\partial_i\phi\}_j$ and $\{(g_j(u)\cdot\partial_i u)\phi \}_j$, as well as the fact that $\partial_i u=0$ a.e. on the set $\{u=0\}$, we conclude again from the Vitali convergence theorem that
$$\int_B G(u)\partial_i\phi=-\int_B(g(u)\cdot\partial_i u)\phi .$$

\end{proof}

\begin{lemma}\label{lemf}
Let $\lambda$ be as in Lemma \ref{lem5}. Then, $\lambda>0$ and $u$ is a distributional solution of the equation $\Delta^2 u =\frac{1}{\lambda} g(u)$,  belonging to $ W^{4,q}_{loc}(\R^4)$, $\forall q\in [1,\infty)$. 
\end{lemma}

\begin{proof}
The linear functional $M(\phi ) =\int_{\R^4} g(u) \cdot \phi - \lambda L(\phi )$ restricted to $C^{\infty}_0$ satisfies $|M(\phi)|\leq C_G \|\phi\|_{L^1}$ (cf. \eqref{eqlem5}), thus by the Riesz Representation theorem, there exists $h\in L^{\infty}(\R^4)$, $h:\R^4\to\R^m$, such that $M(\phi ) =\int_{\R^4} h \cdot \phi$, $\forall \phi\in C^\infty_0$. Furthermore, \eqref{eqlem5} implies that $|\int_{\R^4} h \cdot \phi|\leq C_G\int_{u=0}|\phi|$, for all $\phi\in C^\infty_0$, and thus for all $\phi\in L^1(\R^4)$. Consequently, $h$ vanishes on $\{u\neq 0\}$, and denoting by $\chi$ the characteristic function, we have $-\lambda\Delta^2 u + g(u) = h \chi_{u=0}$ in  $\mathcal D'$. Our next claim is that $\lambda>0$. Indeed, if we assume that $\lambda=0$, then we obtain successively $g(u)\equiv 0$, and $\nabla (G(u))\equiv 0$ (by Lemma \ref{lem6}), from which it follows that $G(u)\equiv 0$, since $G(u)\in L^1(\R^4)$. In particular we would have either $u(x)=0$ or $|u(x)|\geq\epsilon$, for a.e. $x\in \R^4$ (cf. \eqref{g4}), and $\nabla (\min(|u|,\epsilon)\equiv 0$ in $\R^4$, which implies that the function $\min(|u|,\epsilon)$ is constant. This is excluded, since $u\in \mathcal C$ and $u\neq 0$. Hence, it holds $\Delta^2 u =\frac{1}{\lambda} (g(u)-h \chi_{u=0})$ in  $\mathcal D'$, with $\lambda>0$. Finally, we check that given $q>1$ and a bounded set $Q\subset\R^4$, we have $g(u)\in L^q(Q)$. To see this, we use the fact that $\int_{\R^4}A(k_1(|u|-\epsilon/2 )^+)<\infty$, for a constant $k_1>0$ (since $(|u|-\epsilon/2 )^+\in W^{1,4}(\R^4)$), together with the bounds $|g(u)|\leq k_2+ A(\frac{k_1}{q}(|u|-\epsilon/2)^+)$ and $|A(|u|)|^q\leq k_3+A(q|u|)$ holding in $\R^m$ for constants $k_2,k_3>0$. Therefore, we have $u\in  L^q_{loc}(\R^4)$ as well as $g(u)-h \chi_{u=0}\in L^q_{loc}(\R^4)$, $\forall q>1$, and in view of Lemma \ref{reg} we deduce that $u\in W^{4,q}_{loc}(\R^4)$, $\forall q> 1$. This implies that  $\Delta^2u=0$ a.e. on $\{u= 0\}$, and $h=0$ a.e. on $\{u =0\}$. That is, we have
$\Delta^2 u =\frac{1}{\lambda} g(u)$ in  $\mathcal D'$, with $\lambda>0$. 
\end{proof}
 
Finally, Pohozaev identity (cf. Lemma \ref{poz}) implies that $\int_{\R^4}G(u)=0$, while in view of the weak convergence $D^2 u_j \rightharpoonup D^2 u$ in $L^2(\mathbb{R}^4)$, we have $\frac{1}{2}\int_{\R^4}|D^2u|^2\leq T$. Thus, $\frac{1}{2}\int_{\R^4}|D^2u|^2= T$, and $u$ solves the minimization problem \eqref{minimizing set}.  By setting $\bar u(x)=u(\lambda^{-\frac{1}{4}}x)$, we obtain a solution of \eqref{system}, such that $\int_{\R^4}G(\bar u)=0$, and  $\frac{1}{2}\int_{\R^4}|D^2\bar u|^2= T$ (i.e. $\bar u$ still solves the minimization problem \eqref{minimizing set}). Next, let $v\in \mathcal C$ be a solution of \eqref{system} in $\mathcal D'$ such that $v\neq 0$. Then, it holds $\int_{\R^4}G(v)=0$ in view of Pohozaev identity, and we have $S(v)=K(v)\geq T=K(\bar u)=S(\bar u)$. That is, $\bar u$ is a ground state solution. We also point out that the asymptotic convergence of $\bar u$ and $\nabla \bar u$ to $0$ follows from Lemma \ref{poz}.
\begin{remark}\label{rem111}
Any solution $u$ of the minimization problem \eqref{minimizing set} solves the equation $\Delta^2 u =\frac{1}{\lambda} g(u)$ in $\mathcal D'$, where 
$\lambda>0$ is the Lagrange multiplier provided by Lemma \ref{lem5}. In addition, there holds $\int_{\R^4}G(u)=0$.  Indeed, by considering the minimizing sequence $u_j=u$, $\forall j$, the arguments of Lemmas \ref{lemma corresponding to 3.2}, \ref{lem2}, \ref{lem4}, \ref{lem5}, \ref{lem6} and \ref{lemf} can be reproduced and apply to the minimizer $u$. Thus, the map $\bar u(x)=u(rx)$, with $r=\lambda^{-\frac{1}{4}}$, is a ground state solution to \eqref{system}.
\end{remark}

\section{Proof of Theorem \ref{th2}}\label{sec:sec3}
Given the potential $G(u)=|u|^2\ln|u|$ (which clearly satisfies assumptions (\ref{g2})-(\ref{g6})), we know by Remark \ref{remf} that  $\mathcal C=\{u\in H^2(\R^4): \, G^-(u)\in L^1(\R^4)\}$. Let $\mathcal C'=\{u \in \mathcal C: \int_{\R^4}|u|^2=1\}$, let $u\in\mathcal C'$, and set $\mu=e^{-V(u)}$, with $V(u)=\int_{\R^4}|u|^2 \ln
|u|$. We compute $V(\mu u)=\mu^2(\ln(\mu)  \|u\|^2_{L^2(\R^4)}+V(u))=0$, since $\|u\|^2_{L^2(\R^4)}=1$. Consequently, we have the inequality 
$\frac{\mu^2}{2}\int_{\R^4}|D^2 u|^2\geq  T$, which is equivalent to \eqref{sobin}. Next, assume that $u = \bar u/\|\bar u\|_{L^2(\R^4)}$, where $\bar u$ is a ground state solution to \eqref{system}. Then, setting $\mu=1/\|\bar u\|_{L^2(\R^4)}$  one can see that the left hand side of \eqref{sobin} is equal to $\frac{1}{2}\ln(\mu^2)$, while the right hand side is $V(\mu \bar u)=\mu^2(\ln(\mu)  \|\bar u\|^2_{L^2(\R^4)}+V(\bar u))=\ln(\mu)$, since $V(\bar u)=0$. This is an equality case in \eqref{sobin}. Conversely, if the equality in \eqref{sobin} holds for $u\in \mathcal C'$, i.e. $K(u)=T e^{2V(u)}$ (where $K$ is defined in \eqref{action}), then we have $K(\mu u)=T\Leftrightarrow \mu=e^{-V(u)}$. Moreover, for this value of $\mu$ there holds $V(\mu u)=\mu^2(\ln(\mu)  \|u\|^2_{L^2(\R^4)}+V(u))=0$, that is, $\mu u$ solves the minimization problem \eqref{minimizing set}. Therefore, in view of Remark \ref{rem111}, there exists a unique $r>0$ such that $\bar u(x)=\mu u(rx)$ is a ground state solution to \eqref{system}.

To complete the proof of the theorem, it remains to show that  $\frac{1}{2T}<\frac{1}{(2\pi e)^2}$.
To this end, we use \eqref{sobin2} (which also holds in the vector case since we have $|\nabla|u||\leq|\nabla u|$, $\forall u\in H^1(\R^4;\R^m)$), combined with the inequality (cf. \cite[Lemma 6.1]{med3}):
\begin{align}\label{sobin3}
\int_{\R^4}|\nabla u|^2<\Big[\int_{\R^4}|\Delta u|^2\Big]^\frac{1}{2}\Big[\int_{\R^4}|u|^2\Big]^\frac{1}{2}, \text{ for } u\in H^2(\R^4;\R^m), u \neq 0.
\end{align}
In view of Lemma \ref{lap}, we have $\|D^2 u\|_{L^2(\R^4)}=\|\Delta u\|_{L^2(\R^4)}$, thus we obtain 
\begin{align}\label{sobin4}
\dfrac{1}{2} \ln\left(\dfrac{1}{(2\pi e)^2}\int_{\R^4}|D^2 u|^2\right)> \int_{\R^4}|u|^2\ln|u| ,\text{ for } u\in \mathcal C',
\end{align}
and choosing $u = \bar u/\|\bar u\|_{L^2(\R^4)}$ with $\bar u$ a ground state solution to \eqref{system}, we conclude that
\begin{align}\label{sobin5}
\dfrac{1}{2} \ln\left(\dfrac{1}{(2\pi e)^2}\int_{\R^4}|D^2 u|^2\right)> \int_{\R^4}|u|^2\ln|u|=
\dfrac{1}{2} \ln\left(\dfrac{1}{2T}\int_{\R^4}|D^2u|^2\right).
\end{align}
This proves that $\frac{1}{2T}<\frac{1}{(2\pi e)^2}$.

\section{Some lemmas applicable to the solutions of \eqref{system}}\label{sec:sec4}

Here, we gather some lemmas applicable to the solutions of \eqref{system}. We first recall a regularity result for the nonhomogeneous biharmonic equation that can be derived from the $L^p$ estimates in \cite[(2.6)]{nir}.

\begin{lemma}\label{reg}
Let $1 < p < \infty$, and $\Omega \subset\R^n$ an open set. If  $u\in  L^p_{loc}(\Omega)$ and $\Delta^2 u \in  L^p_{loc}(\Omega)$, then $u\in W^{4,p}_{loc}(\Omega)$.
\end{lemma}
Next, we examine the regularity of solutions of \eqref{system} belonging to $\mathcal C$, and state the corresponding Pohozaev identity. This identity was established in \cite[Theorem 1.2]{med3} for dimensions $d\geq 5$. Its proof is identical in dimension $4$.  
\begin{lemma}\label{poz} Assuming that assumptions (\ref{g2})-(\ref{g6}) hold, let $v\in \mathcal C$ be a ditributional solution of \eqref{system}. Then,
we have 
\begin{itemize}
\item[(i)]  $v \in C^{3,\alpha}_{loc}\cap W^{4,q}_{loc}$ for any $0<\alpha<1$ and  $1\leq q<\infty$,
\item[(ii)] $\lim_{|x|\to\infty}v(x)=\lim_{|x|\to\infty}|\nabla v(x)|=0$,
\item[(iii)] $\int_{\R^4}G(v)=0$ (Pohozaev identity).
\end{itemize}
\end{lemma}

\begin{proof}
(i) Let $v\in\mathcal C$ be a ditributional solution of \eqref{system}. Since, $\mu(\{|v|>\epsilon/2\})<\infty$, we can see as in the proof of Lemma \ref{Lemma 3.1}, that inequality \eqref{inequality of lemma 3.1} applies to $v$, for every $q>1$. In particular, given $q>1$, the norms $\|v\|_{L^q(B(x,1))}$ are uniformly bounded on unit balls ($\forall x\in\R^4$). Similarly, the norms $\|g(v)\|_{L^q(B(x,1))}$ are uniformly bounded on unit balls. To see this, we use the fact that $\int_{\R^4}A(k_1(|v|-\epsilon/2 )^+)<\infty$, for a constant $k_1>0$ (since $(|v|-\epsilon/2 )^+\in W^{1,4}(\R^4)$), together with the bounds $|g(v)|\leq k_2+ A(\frac{k_1}{q}(|v|-\epsilon/2)^+)$ and $|A(v)|^q\leq k_3+A(qv)$ holding in $\R^m$ for constants $k_2,k_3>0$. Therefore, it follows from Lemma \ref{reg} that $v\in W^{4,q}_{loc}$, $\forall q> 1$, and from the Sobolev embeddings that  $v \in C^{3,\alpha}_{loc}$, $\forall \alpha\in (0,1)$. 

(ii) Let $w=\Delta v\in L^2(\R^4)$. Another consequence of what precedes is that the norms $\|w\|_{W^{2,2}(B(x,1/2))}$ are uniformly bounded ($\forall x\in\R^4$). This follows from the interior estimates of \cite[Theorem 4.9]{gia} applied to $w$ and $\Delta w=g(v)$, since we have seen that the norms $\|g(v)\|_{L^2(B(x,1))}$ are uniformly bounded. Next, an application of \cite[Theorem 4.11]{gia} to $v$ and $\Delta v=w$, shows that the norms $\|v\|_{W^{4,2}(B(x,1/4))}$ are uniformly bounded ($\forall x\in\R^4$), since the norms $\|v\|_{L^2(B(x,1))}$ are uniformly bounded. Therefore, we deduce from the Sobolev embeddings that $v$ as well as its first derivatives are bounded and uniformly continuous on $\R^4$. 

To prove that $v$ converges asymptotically to $0$, we first establish that the set $\{|v|>\epsilon\}$ is bounded. Indeed, if $|v(x_n)|>\epsilon$ holds for a sequence $\{ x_n\}\subset\R^4$ such that $\lim_{n\to\infty}|x_n|=\infty$, then in view of the uniform continuity of $v$, there exists $\delta>0$ such that $|v|>\epsilon/2$ holds on the balls $B(x_n,\delta)$, which by passing to a subsequence can be assumed to be disjoint. Consequently, we obtain that $\mu(\{|v|>\epsilon/2\})=\infty$, a contradiction. Now, suppose that there exists a sequence $\{ x_n\}\subset\R^4$ such that $\lim_{n\to\infty}|x_n|=\infty$, and $|v(x_n)|>\eta$, for some $\eta\in (0,\epsilon)$. According to what precedes, we have $|v(x_n)|\in [\eta,\epsilon]$ (for $n$ large enough). Then, in view of the uniform continuity of the function $\R^4\ni x\mapsto G(v(x))$ and \eqref{g4}, there exists $\delta >0$, such that $|G(v)|>\beta>0$ holds on the balls $B(x_n,\delta)$, for some $\beta>0$. As previously, the balls $B(x_n,\delta)$ can be assumed to be disjoint, thus we get $G(v)\notin L^1(\R^4)$, a contradiction. Therefore, we proved that $\lim_{|x|\to\infty}v(x)=0$. In a similar way, one can see that $\lim_{|x|\to\infty}|\nabla v(x)|=0$.
 
(iii)  To derive Pohozaev identity, reproduce the arguments in the proof of \cite[Theorem 1.2]{med3}.
\end{proof} 

Finally, we establish that $D^2u$ is interchangeable with $\Delta u$ in the expression of functional $K$.
\begin{lemma}\label{lap} Let $u\in L^1_{loc}(\R^4)$ be such that $\nabla u \in L^4(\R^4)$, and $D^2 u\in L^2(\R^4)$. Then, we have $\int_{\R^4}|D^2u|^2=\int_{\R^4}|\Delta u|^2$. If in addition $u\in L^2(\R^4)$, then it holds $\nabla u\in L^2(\R^4)$, and $\int_{\R^4}|\nabla u|^2=-\int_{\R^4}u\cdot\Delta u$.
\end{lemma}

\begin{proof}
(i) Let $\phi\in  C^\infty(\R^4)$ be such that $\nabla \phi \in L^4(\R^4)$, and $D^2 \phi\in L^2(\R^4)$. By integrating by parts we obtain
\begin{align*}
&\int_{\R^4}|D^2\phi|^2=\lim_{r\to\infty}\sum_{i,j=1}^4\int_{B_r}\partial^2_{ij}\phi\cdot\partial^2_{ij}\phi\\
=&\lim_{r\to\infty} \sum_{i,j=1}^4\Big[ -\int_{B_r}\partial_{i}\phi\cdot\partial^3_{ijj}\phi+\int_{\partial B_r} \nu_j\partial_i \phi\cdot \partial^2_{ij}\phi \Big]\\
=&\lim_{r\to\infty}\sum_{i,j=1}^4\Big[ \int_{B_r}\partial^2_{ii}\phi\cdot\partial^2_{jj}\phi+\int_{\partial B_r} (\nu_j\partial_i \phi\cdot \partial^2_{ij}\phi -\nu_i\partial_i \phi\cdot \partial^2_{jj}\phi)\Big]\\
=&\int_{\R^4}|\Delta\phi|^2+\lim_{r\to\infty}\sum_{i,j=1}^4\Big[\int_{\partial B_r} (\nu_j\partial_i \phi\cdot \partial^2_{ij}\phi -\nu_i\partial_i \phi\cdot \partial^2_{jj}\phi)\Big],
\end{align*}
where $B_r=\{x\in\R^4: |x|<r\}$, and $\nu=(\nu_1,\nu_2,\nu_3,\nu_4) $ is the outward unit normal vector at $\partial B_r$. Moreover, since $|\nabla \phi||D^2\phi|\in L^{\frac{4}{3}}(\R^4)$, there exists a sequence $r_k\to\infty$, such that 
$\lim_{k\to\infty}r_k\int_{\partial B_{r_k}}(|\nabla \phi||D^2\phi|)^\frac{4}{3}=0$. Therefore, we conclude that
\begin{align*}
\Big|\int_{\partial B_{r_k}} (\nu_j\partial_i \phi\cdot \partial^2_{ij}\phi -\nu_i\partial_i \phi\cdot \partial^2_{jj}\phi)\Big|\leq \Big[\int_{\partial B_{r_k}}(|\nabla \phi||D^2\phi|)^\frac{4}{3}\Big]^\frac{3}{4}O(r_k^\frac{3}{4})\to 0,
\end{align*}
as $k\to\infty$, and $\int_{\R^4}|D^2\phi|^2=\int_{\R^4}|\Delta\phi|^2$. Now, if $u\in L^1_{loc}(\R^4)$ is such that $\nabla u \in L^4(\R^4)$, and $D^2 u\in L^2(\R^4)$, we consider a sequence $\rho_n$ of mollifiers. Setting $\phi_n=u\star \rho_n\in C^\infty(\R^4)$, we have $\partial_i \phi_n=\partial_i u\star \rho_n\to \partial_i u$ in $L^4(\R^4)$, as well as $\partial^2_{ij} \phi_n=\partial^2_{ij} u\star\rho_n\to \partial^2_{ij} u$ in $L^2(\R^4)$. Consequently, it holds $\int_{\R^4}|D^2\phi_n|^2=\int_{\R^4}|\Delta\phi_n|^2$, and as $n\to\infty$, we get $\int_{\R^4}|D^2u|^2=\int_{\R^4}|\Delta u|^2$.

(ii) To prove the second statement of the lemma, we consider a function $\phi\in  C^\infty(\R^4)\cap L^2(\R^4)$ such that $\nabla \phi \in L^4(\R^4)$, and $D^2 \phi\in L^2(\R^4)$. As previously, we integrate by parts and obtain
$$\int_{B_r}|\nabla \phi|^2=-\int_{B_r}\phi\cdot\Delta\phi+\int_{\partial B_r} \frac{\partial\phi}{\partial\nu}\cdot \phi.$$
On the one hand, it is clear that $\lim_{r\to\infty}\int_{B_{r}}\phi\cdot\Delta\phi=\int_{\R^4}\phi\cdot\Delta\phi$.
On the other hand, we notice that since $|\nabla \phi||\phi|\in L^{\frac{4}{3}}(\R^4)$, there exists a sequence $r_k\to\infty$, such that 
$\lim_{k\to\infty}r_k\int_{\partial B_{r_k}} (|\nabla \phi||\phi|)^\frac{4}{3}=0$. 
Therefore, we have 
\begin{align*}
\Big|\int_{\partial B_{r_k}} \frac{\partial\phi}{\partial\nu}\cdot \phi\Big|\leq \Big[\int_{\partial B_{r_k}}(|\nabla \phi||\phi|)^\frac{4}{3}\Big]^\frac{3}{4}O(r_k^\frac{3}{4})\to 0,
\end{align*}
as $k\to\infty$, and $\lim_{k\to\infty}\int_{B_{r_k}}|\nabla \phi|^2=-\int_{\R^4}\phi\cdot\Delta\phi$. This proves that $\nabla \phi \in L^2(\R^4)$, and $\int_{\R^4}|\nabla\phi|^2=-\int_{\R^4}\phi\cdot\Delta\phi$.
Finally, if $u\in L^2(\R^4)$ is such that $\nabla u \in L^4(\R^4)$, and $D^2 u\in L^2(\R^4)$, we proceed as previously by mollification. Setting $\phi_n=u\star \rho_n\in C^\infty(\R^4)$, we have that $ \phi_n\to u$ as well as $\Delta \phi_n=\Delta u\star\rho_n\to \Delta u$ in $L^2(\R^4)$, while $\partial_i \phi_n=\partial_i u\star \rho_n\to \partial_i u$ in $L^4(\R^4)$ and a.e. in $\R^4$ (up to subsequence). Consequently, it holds $\int_{\R^4}|\nabla\phi_n|^2=-\int_{\R^4}\phi_n\cdot\Delta\phi_n$, and as $n\to\infty$, we get that $\nabla u\in L^2(\R^4)$, and $\int_{\R^4}|\nabla u|^2=-\int_{\R^4}u\cdot \Delta u$.

\begin{remark}\label{remf} It follows from Lemma \ref{lap} that given the potential $G(u)=|u|^2\ln|u|$, we have $\mathcal C=\{u\in H^2(\R^4): \, G^-(u)\in L^1(\R^4)\}$. Indeed, if $u\in\mathcal  C$, it is clear that $u\in L^2(\R^4)$, since $G(u)\in L^1(\R^4)$, thus Lemma \ref{lap} implies that $\nabla u\in L^2(\R^4)$, and $u\in H^2(\R^4)$. Conversely, if $u\in H^2(\R^4)$, we have $u,\,  \nabla u\in H^1(\R^4)\subset L^4(\R^4)$, and $\mu(\{|u|>\eta\})<\infty$, $\forall \eta>0$, since $u\in L^2(\R^4)$. Moreover, it holds $G^+(u)\in L^1(\R^4)$, in view of the bound $G^+(u)\leq|u|^4$. 
\end{remark}

\end{proof}
\section{Conclusion}
The method of constrained minimization applies to prove the existence of a ground state solution for both systems \eqref{systemp} and  \eqref{system}, in  space dimensions greater or equal than that of the limiting case (i.e. $d\geq 2$ for  \eqref{systemp}, and  $d\geq 4$ for \eqref{system}). In dimension $d=2$,  the ground state solution of the scalar second order equation \eqref{systemp} ($m=1$) was constructed in \cite{BGK}, assuming that the potential $G$ satisfies an exponential growth condition. This hypothesis is required to obtain compactness properties from the Trudinger inequality. Similarly, for the fourth order system \eqref{system} in dimension $d=4$, an exponential growth condition is assumed for $G$. We also point out that if the space dimension is smaller than that of the limiting case, then the problem of finding nontrivial solutions converging asympototically to $0$ is different. In space dimension $d=1$, such solutions of \eqref{systemp} are the homoclinic orbits  (cf. for instance \cite{antonop}). On the other hand, the construction of such solutions of \eqref{system} in space dimensions $d=1,2,3$, is to a large extent an open problem.

\section*{Acknowledgements}
This research project is implemented within the framework of H.F.R.I call
``Basic research Financing (Horizontal support of all Sciences)'' under the National Recovery and Resilience Plan ``Greece 2.0'' funded by the European Union - NextGenera tionEU (H.F.R.I. Project Number: 016097).

\end{document}